\documentclass{amsart}
\usepackage{amssymb,latexsym}
\usepackage{amsmath}
\usepackage[dvipdfm]{graphicx}
\usepackage{enumerate}
\newenvironment{enumeratei}{\begin{enumerate}[\upshape (i)]}%
                            {\end{enumerate}}
 \usepackage{color}
 \def\piros#1{#1}
 \def\vajon#1{}
 \def\sajat#1{}

\theoremstyle{plain}
 \newtheorem{theorem}{Theorem}[section]
 \newtheorem{lemma}[theorem]{Lemma}
 \newtheorem{proposition}[theorem]{Proposition}
 \newtheorem{corollary}[theorem]{Corollary}
\theoremstyle{definition}
  
  \newtheorem{notation}[theorem]{Notation}
  
  \newtheorem{example}[theorem]{Example}
  
\renewcommand \epsilon{\varepsilon}
\newcommand \cnvclose [2]{\textup{Cnv}_{#1}( #2 )}
\newcommand \tline[3]{\ell_{#1}(#2,#3)}
\newcommand \lsegmnt[5] {[#4,#5]_{\tline{#1}{#2}{#3}}}
\newcommand \keplet[4] {\Phi^{\opset(#1)}_{#2}(#3;#4)}
\newcommand \ees {\mathrel{\&}}
\newcommand\opint{I^{o}}
\newcommand\boopint{I^{{\bullet}}}

\newcommand\opset{\underline{I}^{o}}
\newcommand\boopset{\underline{I}^{{\bullet}}}
\newcommand\gap{\kern 0.5 pt}
\newcommand\gapt{\kern 1.5 pt}
\newcommand\moop[1] {\underline{#1}}
\newcommand\mop[1] {\gap{\underline {\gap #1 }}\gap}
\newcommand\hop{\mop{h}}
\newcommand\hoop{\moop{h}}
\newcommand\pop{\mop{p}}
\newcommand\poop{\moop{p}}

\newcommand\qop{\mop{q}}
\newcommand\qoop{\moop{q}}
\newcommand\rop{\mop{r}}
\newcommand\sop{\mop{s}}
\newcommand \aspace [2] {{\textup{Aff}_{#1}({#2})}}
\newcommand \adim [2] {\textup{dim}^{\scriptscriptstyle{\textup{aff}}}_{#1}(#2)}
\newcommand \relratpart [2] {\textup{rr}_{#1}(#2)}
\newcommand \rrps[1] {\relratpart{\vec a}{#1}}
\newcommand \rrpv[1] {\relratpart{\vec a'}{#1}}
\newcommand \topclose [2] {[#2]_{#1}^{\scriptscriptstyle{\textup{top}}}}

\newcommand\pontf[1] {{#1}^{\kern-0.6pt \ast}}
\newcommand\restrict[2]{#1\rceil_{#2}}
\newcommand \distance [2] {\textup{dist}(#1,#2)}
\newcommand \diam [1] {\textup{diam}(#1)}
\newcommand \aspan [2] {\textup{Span}_{#1}^{ \scriptscriptstyle{\textup{aff}}}(#2)}

\newcommand \toT [1] {\mathrel{\to_{(#1,\opset(T))}} }

\newcommand \ttau {{\pmb{\tau}}}

\newcommand\fsub[1]{\langle #1\rangle_{\kern-1pt\textup{field}}}

\newcommand\vspan [2]{{}_{#1}\kern-1.2pt\langle #2\rangle }

\newcommand\vconon[1] {\mathord{\sim}\kern-2pt{{\ast}\atop{\kern-1.5pt #1}}}

\newcommand \set[1] {\{#1\}}
\renewcommand\phi{\varphi}

\newcommand\iffform[1] {\mathrel{\mathord{\iff}^{\kern-3pt\eqref{#1}}}}
\newcommand\iffforms[2] {\mathrel{\mathord{\iff}^{\kern-3pt\eqref{#1},\eqref{#2}}}}
\newcommand\iffindhyp {\mathrel{\mathord{\iff}^{\kern-3pt ih}}}
\newcommand\eqindhyp {\mathrel{\mathord{=}^{\kern0pt ih}}}
\newcommand\concat{\cup^{\kern-5pt\ast}}
\newcommand\Knclos[3]{\tilde K^{(#1)}_{#2}\kern -1pt (#3)}
\newcommand\Kfclos[2]{K^{(1)}_{#1}\kern -1pt (#2)}
\newcommand\Kffclos[2]{K^{(1)}_{#1}\kern -1pt \bigl(#2\bigr)}
\newcommand\closure[2]{K^{\kern-0.0pt\infty}_{#1}\kern-0.7pt(#2)}
\newcommand\fullclosure[1]{K^{\kern-0.0pt\infty}_{\textup{all}}\kern-0.7pt(#1)}
\newcommand\fullbigclosure[1]{K^{\kern-0.0pt\infty}_{\textup{all}}\kern-0.7pt\bigl(#1\bigr)}

\renewcommand\rho{\varrho}

\newcommand\hlcub[2]{g^{(#1)\kern -1pt L}_{#2}}
\newcommand\hrcub[2]{g^{(#1)\kern -1pt R}_{#2}}
\newcommand\hwcub[3]{g^{(#2)\kern -1pt #1}_{#3}}

\newcommand\plcub[2]{\hat g^{(#1)\kern -1pt L}_{#2}}
\newcommand\prcub[2]{\hat g^{(#1)\kern -1pt R}_{#2}}
\newcommand\pwcub[3]{\hat g^{(#2)\kern -1pt #1}_{#3}}

\newcommand\lrtwo{\mathbf 2_{\kern -1.6pt {\scriptscriptstyle R}}^{\kern-1pt {\scriptscriptstyle L}}}

\begin{document}

\title
{A dyadic view of rational convex sets}
\author[G.\ Cz\'edli]{G\'abor Cz\'edli}
\email{czedli@math.u-szeged.hu}
\urladdr{http://www.math.u-szeged.hu/$\sim$czedli/}
\address{University of Szeged\\Bolyai Institute\\Szeged,
Aradi v\'ertan\'uk tere 1\\HUNGARY 6720}

\author[M.\ Mar\'oti]{Mikl\'os Mar\'oti}
\email{mmaroti@math.u-szeged.hu}
\urladdr{http://www.math.u-szeged.hu/$\sim$mmaroti/}
\address{University of Szeged\\Bolyai Institute\\Szeged,
Aradi v\'ertan\'uk tere 1\\HUNGARY 6720}

\author[{A.\,B.\ Romanowska}]{{A.\, B.\ Romanowska}}
\email{aroman@mini.pw.edu.pl} 
\urladdr{http://www.mini.pw.edu.pl/$\sim$aroman/}
\address{Faculty of Mathematics and Information Sciences, Warsaw University
of Technology, 00-661 Warsaw, Poland}

\thanks{This research was supported by the NFSR of Hungary (OTKA), grant numbers   K77432 and K83219,
\piros{and by the Warsaw University of Technology under
grant number 504G/1120/0054/000}. }

\subjclass[2000]{06C10}
\keywords{convex set, mode, barycentric algebra, commutative medial groupoid, entropic groupoid, entropic algebra, dyadic number}

\date{March 22, 2012}

\begin{abstract} Let $F$ be a subfield of the field $\mathbb R$ of real numbers. Equipped with the binary arithmetic mean operation, each convex subset $C$ of $F^n$ becomes a commutative binary mode, also 
\piros{called idempotent commutative medial (or entropic)
groupoid.} Let $C$ and $C'$ be convex subsets of $F^n$. Assume that they are of the same dimension \piros{and at least one of them is bounded}, or $F$ is the field of all rational numbers. We prove that the corresponding \piros{idempotent} commutative medial groupoids are isomorphic if{f} the affine space $F^n$ over $F$ has an automorphism that maps \piros{$C$} onto \piros{$C'$}. 
We also prove a more general statement for the case when $C,C'\subseteq F^n$ are considered barycentric algebras over a unital subring of $F$  that is distinct from the ring of integers. 
\piros{A related result, for a subring of $\mathbb R$ instead of a subfield $F$, is given in \cite{rczgaroman2}.}
\sajat{Analogous  results for line segments and  polygons over the ring of dyadic numbers were previously proved and intensively used by Matczak, Romanowska and Smith~\cite{MatRomSm:dyadpolyg}.}  
\end{abstract}

\maketitle

\section{Introduction and motivation}\label{sec:intro}
\vajon{
The present paper intends to be self-contained modulo any standard book on Universal Algebra, see the easy-to-reach Burris and Sankappanavar~\cite{rBurSan} for example.} 
Let $F$ be a subfield of the field $\mathbb R$ of real numbers. Equipped with the arithmetic mean operation $(x,y)\mapsto (x+y)/2$, denoted by $\hoop$ (coming from ``half''),  
$F^n$ becomes a groupoid $(F^n,\hoop)$. This groupoid is  idempotent, commutative, medial, and cancellative. In Polish notation, which we use in the paper, these properties mean that, for arbitrary $x,y,z,t\in F^n$,   $xx\hop=x$ (idempotence), $xy\hop=yx\hop$ (commutativity),
$xy\hop\gapt zt\hop\gapt\hop=xz\hop\gapt yt\hop\gapt\hop$ (mediality, \piros{which is a particular case of entropicity}), and $xy\hop=xz\hop$ implies $y=z$ (cancellativity). These groupoids without assuming cancellativity are also called  
\emph{\piros{commutative binary} modes}  \piros{or \emph{CB-modes},}   and they \piros{were} studied in, say,  \cite{MR04}   and \cite{RS91} and \cite{RS02}, and Je\v{z}ek and Kepka~\cite{JK283}.

\piros{Let $C$ be a nonempty subset of $F^n$.}
If there is a convex subset $D$ of the  Euclidean space $\mathbb R^n$ in the usual sense such that $C=D\cap F^n$, then $C$ will be called a \emph{geometric convex subset} of $F^n$. 
\piros{We also say that $C$ is a \emph{geometric convex set over $F$}.}
Later we will give an ``internal'' definition that does not refer to $\mathbb R$. 
\sajat{The most important particular case is when $F$ equals $\mathbb Q$, the field of rational numbers.} 
\piros{Note that $C$ above is simply called a \emph{convex subset} in 
Romanowska and  Smith~\cite{RS02}; however, the adjective ``geometric'' becomes important soon in a more general situation. For convenience, the empty set will not be called a geometric convex set.}

Our initial problem is to characterize those pairs $(C_1,C_2)$ of geometric convex subsets of $F^n$ for which $(C_1,\hoop)$ and $(C_2,\hoop)$ are isomorphic groupoids. In the particular case \piros{when} $F=\mathbb Q$\piros{,} loosely speaking we are interested in what we can see from the ``rational world'' $\mathbb Q^n$ if the only thing we can percept is whether a point equals the arithmetic mean of two other points.
%

%

Similar questions \piros{were} studied for some particular geometric convex subsets of $\mathbb D^2$,
where  $\mathbb D=\set{x2^{k}:x,k\in \mathbb Z}$ is the ring of \emph{\piros{rational} dyadic numbers}. 
Namely, the isomorphism problem of line segments and  \piros{polygons} of the \piros{rational} dyadic plane $\mathbb D^2$ \piros{were} studied in 
Matczak, Romanowska and Smith~\cite{MatRomSm:dyadpolyg}. 
Another problem of deciding whether $(C_1,\hoop)$ is isomorphic to $(C_2,\hoop)$ \piros{is} considered in 
\cite[Ex.\ \piros{2.6}]{rczgaromanI}\piros{, and  \cite{rczgaroman2} also considers a related isomorphism problem.} 

The isomorphism problem even for intervals of the dyadic line $\mathbb D$ is not so evident as one may expect. This explains why our convex sets in the main result, Theorem~\ref{thmmain},  are assumed to have some further properties, including that they are \emph{geometric} \piros{over a subfield of $\mathbb R$}. Further comments on the main result will be given in Section~\ref{SecExCom}.

\section{Barycentric algebras over unital subrings of $\mathbb R$ and the results}
\begin{notation}\label{notattests}
The general assumption and notation in the paper are \piros{the following.} 
\begin{enumeratei}
\item $\mathbb N=\set{1,2,\ldots}$, $\mathbb N_0=\set{0,1,2,\ldots}$, $\mathbb Z$ is the ring of integers, $\mathbb Q$ is the field of rational numbers, $\mathbb R$ is the field of real numbers, and \piros{$n\in \mathbb N$}.
\item $T$ is a subring of $\mathbb R$ such that $1\in T$ and $T\cap\mathbb Q\neq \mathbb Z$ (that is, $\mathbb Z\subset T\cap \piros{\mathbb Q}$). 
\item $K$ is the subfield of $\mathbb R$ generated by $T$, and $F$ is a subfield of $\mathbb R$ such that $T\subseteq F$. (Clearly, $T\subseteq K\subseteq F\subseteq \mathbb R$.)
\item The open and the closed unit intervals of $T$ are denoted by $\opint(T)=\{x\in T: 0< x < 1\}$ and $\boopint(T)=\{x\in T: 0\leq x\leq 1\}$, respectively; $\opint(F)$, $\boopint(\mathbb Q)$, etc.\ are particular cases. (Notice that $T$ can equal, say, $F$ and $F$ can equal $\mathbb R$, etc. Therefore, whatever we define for $T$ or $F$ in what follows, it will automatically make sense for $F$ or $\mathbb R$.)
\item With each $p\in \mathbb R$ we associate a binary operation symbol denoted by $\poop$. For $H\subseteq \mathbb R$, we let $\underline H:=\set{\poop: p\in H}$. However, we will write, say, \piros{$\opset(T)$} instead of $\underline{\opint(T)}$. For $x,y\in\mathbb R^n$, $xy\pop$ is defined to be $(1-p)x+p y$.
\end{enumeratei}
\end{notation}

If $p\in\opint(\mathbb R)$, then $\poop$ is called a \emph{barycentric operation} since $xy\pop$  gives the barycenter  of a two-body system with weight $(1-p)$ in the point $x$ and weight $p$ in the point $y$. For any $p,q$ in $\mathbb R$, the operations $\poop$ and $\qoop$ commute in $\mathbb R^n$, that is  $xy\pop\gapt zt\pop\gapt\qop=xz\qop\gapt yt\qop\gapt\pop$ holds for all $x,y,z,t\in\mathbb R$. \piros{This property is called the \emph{entropic law},  see \cite{RS02}}. \piros{As a particular case, the \emph{medial law} (for $\hoop$) means that $\hoop$ commutes with itself}. 
Although the present \piros{paper} is more or less  self-contained, for standard general algebraic concepts the reader may want to see Burris and Sankappanavar~\cite{rBurSan}. \piros{He may also want to see Romanowska and Smith~\cite{RS02} for additional information on modes and barycentric algebras.}
The visual meaning of  barycentric operations is revealed by the following lemma; the obvious proof will be omitted. The Euclidean distance  $\bigl((x_1-y_1)^2+\cdots+ (x_n-y_n)^2\bigr)^{1/2}$ of  $x,y\in\mathbb R^n$ will be denoted by $\distance xy$. 

\begin{lemma}\label{barireveAl} Let $y$ and $x$ be distinct point\piros{s} in $\mathbb R^n$, see Figure~\ref{figone}. Then for each $b$ belonging to the open line segment connecting $y$ and $x$ and for each $p\in\opint(R)$, 
\begin{align*}
b=yx\pop\iff x=yb\mop{1/p}\iff y=bx\,\mop{p/(p-1)}\text.
\end{align*}
Moreover, $\distance yx =\distance yb / p$. 
\end{lemma}

\begin{figure}
\centerline
{\includegraphics[scale=1.0]{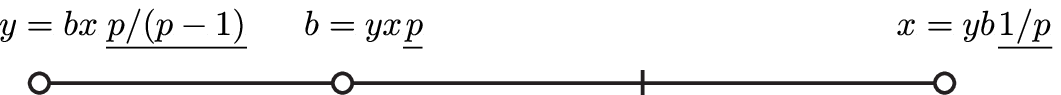}}
\caption{Illustrating Lemma~\ref{barireveAl}  in case  $p=1/3$}\label{figone}
\end{figure}

\piros{The algebra $(\mathbb R^n; \opset(T))$ and all of its subalgebras are particular 
members of the variety of 
barycentric algebras over $T$, or \emph{$T$-barycentric algebras} for short.
(However, as opposed to previous papers and monographs,  $T$ is no longer assumed to be a field.) These particular  
$T$-barycentric algebras that we consider} are \emph{modes}, that is, idempotent algebras in which any two operations (and therefore any two term functions) commute. 
Modes and barycentric algebras have intensively been studied in the monographs \cite{RS85} and \cite{RS02},  see also
the extensive bibliography in \cite{rczgaromanI}.
It is well-known, see \cite{RS02}, that  $(F^n; \hoop)$ is term-equivalent to  $(F^n; \opset(\mathbb D))$, whence the same holds for its subalgebras. This allows us to translate the initial problem to the language of $\mathbb D$-barycentric algebras, and then it is natural  to extend it to $T$-barycentric algebras.

The subalgebras of  $(\mathbb R^n; \opset(T))$ will be called \emph{$T$-convex subsets} of $\mathbb R^n$. \piros{The empty set is not considered to be $T$-convex.  (Notice that the adjective ``$T$-convex'' in \cite{rczgaroman2} is used only for subsets of $T^n$.)}
For $\emptyset\neq X\subseteq \mathbb R^n$, the \emph{$T$-convex hull} of $X$, denoted by $\cnvclose TX$, is the subalgebra  generated by $X$ in  $(\mathbb R^n; \opset(T))$. It is well-known, see \cite{RS02}, that  $\boopset(T)$ is exactly the set of binary term functions of $(F^n; \opset(T))$. Moreover, each $(1+k)$-ary term function of $(F^n; \opset(T))$ agrees with a function $\ttau \colon (x_{\piros{0}},\ldots,x_k)\mapsto \xi_0x_0+\cdots+\xi_kx_k$ where $\xi_0,\ldots,\xi_k\in \boopint(T)$ such that $\xi_0+\cdots+\xi_k=1$. 
This implies that, for any $\emptyset\neq X\subseteq F^n$, 
\begin{equation}\label{HoWgEnre}
\cnvclose TX=\set{x_0\cdots x_k\,\ttau: k\in\mathbb N_0,\text{ } x_0,\ldots,x_k\in X\text{ and }\ttau\text{ is as above}}\text.
\end{equation} 
The full idempotent reduct of the  $T$-module $_TF^n$ is a so-called affine module \piros{over $T$}; we call it an   \emph{affine $T$-module} and denote it by $\aspace T{F^n}$. When $T$ is understood or irrelevant, we often write 
$F^n$ instead of $\aspace T{F^n}$.
 In the particular case $T=F$, the \piros{affine $F$-module}   
 $\aspace F{F^n}$ is an $n$-dimensional \emph{\piros{affine $F$-space}}, see more (well-known) details later.

The mere assumption that $C\subseteq F^n$ is a $T$-convex subset would rarely be sufficient \piros{for} our purposes, \piros{see also \cite{rczgaroman2} for a similar analysis}.  There are three reasonable ways to make a stronger assumption. 

Firstly, we can assume that $C$ is an \piros{\emph{$F$-convex subset}, that is, a subalgebra of $(F^n,\opset(F)$).}

Secondly, we can assume that $C$ is the intersection of $F^n$ with an $\mathbb R$-convex subset  of $\mathbb R^n$. (That is, with a convex subset of $\mathbb R^n$ in the usual geometric meaning.) In this case we say that $C$ is a \emph{\piros{geometric convex subset}} of $F^n$. \piros{In other words, we say that $C$ is a \emph{geometric convex set over $F$.}  
Notice that
the geometric convexity of $C$  depends on $F$, so we can use this concept only for subsets of  $F^n$. (Note also that \cite{rczgaroman2} defines geometric convexity even when   $C\subseteq T^n$ and it does it in a different way, which is equivalent to our approach for the case $T=F$.)}

To define the third variant of convexity, let $a,b\in F^n$ with $a\neq b$. By the \emph{$T$-line} generated by $\set{a,b}$ we mean 
\piros{$\tline Tab:=  \cnvclose T {a,b}$. By \eqref{HoWgEnre},} we have that $\tline Tab =   \set{ab\pop: p\in \piros T}$.
It follows from cancellativity that for each $x\in\tline Tab$, there is exactly one $p\in T$ such that $x=ab\pop$. Let $c,d\in\tline Tab$. Then there are unique $p,r\in T$ such that $c=ab\pop$ and $d=ab\rop$. For $s\in T$, we say that $s$ is between $p$ and $r$ if{f} $p\leq s\leq r$ or $r\leq s\leq p$. Then 
\[\lsegmnt Tabcd:=\set{ab\sop: s\text{ is between }p\text{ and }r}
\]
is called a \emph{$T$-segment} of the $T$-line $\tline Tab$ with \emph{endpoints}  $c$ and $d$. \piros{As} opposed to the case when $T$ happens to be a field, a $T$-segment is usually  \emph{not} determined by its endpoints. For example, 0 and 3 are the endpoints of the $\mathbb D$-segment $\lsegmnt {\mathbb D}0103$ and also of the $\mathbb D$-segment $\lsegmnt {\mathbb D}0303$ in $\mathbb Q^1$, but
 $1\in \lsegmnt {\mathbb D}0103\setminus \lsegmnt {\mathbb D}0303$ indicates that
these $\mathbb D$-segments are distinct.
Now, a \piros{nonempty} subset $\piros C$ of $F^n$ will be called \piros{\emph{$T$-segment convex}} if for all $c,d\in C$ and all $T$-segments $S$ with endpoints $c$ and $d$, $S\subseteq C$. {This}  definition, 
is quite ``internal'' since it does not refer to outer objects like $\mathbb R$ \piros{(besides that $T$ is a subring of $\mathbb R$)}.
The relationship among the three concepts \piros{above}  is clarified by the following statement, to be proved later. \piros{A related treatment for analogous concepts is given in \cite{rczgaroman2}.}

\begin{proposition}\label{vknxLl} Let $C$ be a \piros{nonempty}  subset of $F^n$.
\begin{enumeratei}
\item\label{vknxLli} \piros{If $C$ is $T$-segment convex, then it is $T$-convex.}
\item\label{vknxLlii} $C$ is \piros{a}  geometric \piros{convex subset of $F^n$} if{f} it is $F$-convex.
\item\label{vknxLliii} If $C$ is $F$-convex, then it is \piros{$T$-segment convex}.
\item\label{vknxLliv} If \piros{$T$ generates $F$ $($that is, if $F=K)$,} then $C$ is $F$-convex if{f} it is $T$-segment convex.
\end{enumeratei}
\end{proposition}

\piros{Besides \eqref{vknxLli}, each of the conditions \eqref{vknxLlii}--\eqref{vknxLliv} above clearly} implies $T$-convexity. 
\sajat{Hence, instead of speaking of ``$F$-convex $T$-convex'' subsets of $F^n$, we will simply say ``$F$-convex subsets''. }Remember that  $\mathbb Z\subset T\cap \mathbb Q$ means that $\mathbb Z\neq T\cap \mathbb Q$ and $\mathbb Z\subseteq T\cap \mathbb Q$.
If  $X\subseteq F^n$ and 
$\set{\distance xy: x,y\in X}$ is a bounded subset of $\mathbb R$, then  $X$ is called a \emph{bounded} set. 
For $X\subseteq \mathbb R^n$, \piros{the affine \emph{$F$-subspace spanned}} by $X$ will be denoted by $\aspan FX$.
\piros{As usual,} by the \piros{affine \emph{$F$-dimension}} of $X$, denoted by $\adim FX$, we mean the \piros{affine $F$-dimension} of  $\aspan FX$.
We are now in the position to formulate the main result.

\begin{theorem}\label{thmmain} Assume that $n\in\mathbb N$, 
 $F$ is a subfield of $\mathbb R$,  $T$ is a subring of $F$, and $\mathbb Z\subset T\cap \mathbb Q$. Let $C$ and $C'$ be $F$-convex subsets  \piros{$($equivalently, geometric convex subsets$)$}  of $F^n$.  
Assume also that  
\begin{itemize}
\item[\piros{(a)}] $F=\mathbb Q$,
\end{itemize}
or
\begin{itemize}
\item[\piros{(b)}] $C$  and $C'$ 
\piros{have} the same \piros{affine $F$-dimension and at least one of them is bounded.}
\end{itemize}
Then the following three conditions are equivalent. 
\begin{enumeratei}
\item\label{imainitem} $(C,\opset(T))$ and $(C',\opset(T))$ are isomorphic $T$-barycentric algebras.
\item\label{iimainitem} The \piros{affine $F$-space} $\aspace F{F^n}$ has an automorphism $\psi$ such that $\psi(C)=C'$.
\item\label{iiimainitem} The \piros{affine real} space $\aspace{\mathbb R}{\mathbb R^n} $ has an automorphism $\psi$ such that $\psi(C)=C'$.
\end{enumeratei}
\end{theorem}

\begin{corollary}\label{coroLegy} If $C$ and $C'$ are 
\piros{geometric convex} subsets of $F^n$, then $(C,\hoop)\cong(C',\hoop)$ if{f} 
 \eqref{iimainitem} of Theorem~\ref{thmmain}
holds if{f}  \eqref{iiimainitem} of Theorem~\ref{thmmain} holds. Furthermore, if $D$ and $D'$ are \piros{isomorphic subalgebras of $(\mathbb Q^n,\hoop)$},
then $D$ is \piros{a} geometric \piros{convex subset of $\mathbb Q^n$} if{f} \piros{so is} $D'$.
\end{corollary}

\section{Examples and comments}\label{SecExCom}
\piros{Before proving our results, we present four examples to illustrate and comment them.} 
The \piros{first example below} is a variant of \cite[Ex.\ 1.5]{rczgaromanI}. While \cite{rczgaromanI} is insufficient to handle it,  Theorem~\ref{thmmain} will apply easily. Remember that $h$ stands for $1/2$.

\begin{example}\label{exc:one} 
Let $C_i=\{(x,y)\in F^2: x^2\in\opint(F)$ and $|y|<1-|x|^i\}$, for $i\in\mathbb N$. Are there distinct $i,j\in\mathbb N$ such that the groupoids $(C_i,\hoop)$ and $(C_j,\hoop)$ are isomorphic?

The answer is negative. If $(C_i,\hoop)\cong(C_j,\hoop)$, then Theorem~\ref{thmmain} yields an automorphism $\psi$ of $\aspace{\mathbb R}{\mathbb R ^2}$ such that $\psi(C_i)=C_j$. \piros{The usual topological closure of $C_t$ is denoted by $\topclose{\mathbb R}{C_t}$, for $t=1,2$.}
Since $\psi$ is continuous, $\psi(\topclose{\mathbb R}{C_i})=\topclose{\mathbb R}{C_j}$. 
Let $B_t$ denote the boundary  
\[
\text{$\piros{\topclose{\mathbb R}{C_j}\setminus C_t = {} }  \{(x,y)\in \mathbb R^2: -1\leq x\leq 1$ and $|y|\mathrel{\piros{=}}1-|x|^t\}$}
\]
 of $C_t$, for $t=i,j$. \piros{Clearly,}
 $\psi(B_i)=B_j$, \piros{which} is a contradiction since $\psi$ is a linear mapping that preserves the degree of polynomials whose graphs \piros{are}  transformed. 
\end{example}

\begin{example} Let $n=1$, $F=\mathbb Q(\sqrt 2)$, $T=\mathbb D$, and let $C$ be the least $T$\piros{-segment convex} subset of $F=F^n$ that includes $\set{0,3}$. Since $[0,3]\cap\mathbb Q$ is $T$\piros{-segment convex} and includes $\set{0,3}$,
we conclude that $C\subseteq [0,3]\cap\mathbb Q$. Hence  $\sqrt 2\notin C$, and  $C$ is not $F$-convex. 

Thus, the assumption $F= \piros K$ in Proposition~\ref{vknxLl}\eqref{vknxLliv} cannot be omitted.
\end{example}

\begin{example} The rational vector spaces $_{\mathbb Q}(\mathbb R{}\piros{ \times\set 0)}$ and $_{\mathbb Q}\mathbb R^2$ are well-known to be isomorphic since they have the same dimension. (\piros{Recall} that any basis of $_{\mathbb Q}\mathbb R {} \piros{ {} \cong {}_{\mathbb Q}(\mathbb R \times\set 0)}$ is called a \emph{Hamel-basis}.) 
Therefore 
$C=\aspace{\mathbb Q} {\mathbb R {}\piros{ \times\set 0}}$ and $C'=\aspace{\mathbb Q}{\mathbb R^2}$ 
are isomorphic \piros{affine $\mathbb Q$-spaces}.    Thus, $(C,\opset(\mathbb Q))$ is isomorphic to $ (C',\opset(\mathbb Q))$, and they are both $\mathbb R$-convex subsets of $\aspace{\mathbb R}{\mathbb R^2}$. However, no automorphism of $\aspace{\mathbb R}{\mathbb R^2}$ maps $C$ onto $C'$. 

Observe here that $\adim{\piros{\mathbb R}}C\neq \adim{\piros{\mathbb R}}{C'}$, and none of  $C$ and $C'$ is bounded. This motivates  \piros{(without explaining  fully)} the assumption  
``$C$ and $C'$ have the same \piros{affine $F$-dimension and at least one of them is bounded}'' 
in  Theorem~\ref{thmmain}.
\end{example}

\begin{example}
A routine application of Hamel bases shows that the unit disc $(C_1,\hoop):=(\set{(x,y): x^2+y^2<1},\hoop)$ is isomorphic to \piros{another} subalgebra  $(C_2,\hoop)$ of $(\mathbb R^2,\hoop)$ such that both $C_2$ and $\mathbb R^2\setminus C_2$ are everywhere dense in the plane; see 
\cite[Proof of Lemma \piros{2.7}]{rczgaromanI} for details. However, no automorphism of $\aspace{\mathbb R}{\mathbb R^2}$ maps \piros{$C_1$} onto \piros{$C_2$}. 

This motivates the assumption in Theorem~\ref{thmmain} that $C$ and $C'$ are \piros{\emph{geometric convex}} subsets of $F^n$. 

This and the  previous  example show that Theorem~\ref{thmmain} is not valid for arbitrary $T$-convex subsets of $F^n$, so we 
\piros{added} some further assumptions.  However, it remains an open problem whether one could somehow relax the present assumptions. In particular, we do not know  whether they are independent.
\end{example}

\section{Auxiliary statements and proofs}
It is well-known that given an affine space $\piros{V=}\aspace F{V}$, \piros{which is} the full idempotent reduct of the vector space $_FV$, we can obtain the vector space structure back as follows: fix an element $o\in V$, to play the role of 0, define $x+y:=x-o+y$ and, for $p\in F$, $p x:=ox\pop$. This explains some (also well-known) basic facts on affine independence. Namely, 
a $(1+k)$-element subset $\set{a_0,\ldots,a_k}$ of \piros{$\aspace {\piros F}{V}$ is} called \emph{\piros{affine $F$-independent}},  if $a_i\notin\aspan {\piros F}{a_0,\ldots,a_{i-1},a_{i+1},\ldots,a_k}$, for $i=0,\ldots,k$. In this case, each element of the \piros{affine $F$-subspace} $\piros U:=\aspan {\piros F}{a_0,\ldots,a_k}$ can uniquely be written in the form 
$\xi_0a_0+\cdots+\xi_ka_k$ where the so-called \emph{barycentric coordinates} $\xi_0,\ldots,\xi_k$ belong to $F$ and their sum equals 1. Moreover, then $U=\aspace F{U}$ is freely generated by $\set{a_0,\ldots,a_k}$\piros{; that is, each mapping $\set{a_0,\ldots,a_k}\to U$ extends to an endomorphism of $\aspace F{U}$.}

To capture convexity, we need a similar concept: 
$\set{a_0,\ldots,a_k}\subseteq {F^n}$
will be called \emph{$\opset(T)$-independent} if $a_i\notin\cnvclose T{a_0,\ldots,a_{i-1},a_{i+1},\ldots,a_k}$, for $i=0,\ldots,k$. 
It is not hard to see (and it is stated in \cite{PRS03}) that
if $\set{a_0,\ldots,a_k}\subseteq F^n$ is \piros{affine $K$-independent}, then it is a    free generating set of $(\cnvclose{T}{a_0,\ldots,a_k}, \opset(T))$ and of  $(\cnvclose{K}{a_0,\ldots,a_k}, \opset(K))$. However\piros{, as opposed to affine $K$-independence},  $\opset(K)$-independence does not  imply free $\opset(K)$-generation.  For example, the vertices $a_0,\ldots,a_5$ of a regular hexagon in the real plane form an $\opset(\mathbb R)$-independent subset but 
$(\cnvclose{\mathbb R}{a_0,\ldots,a_5}, \opset(\mathbb R))$ is not freely generated since $a_0a_3\hop=a_1a_4\hop$. 

As usual, maximal independent subsets are called \emph{bases}\piros{, or \emph{point bases}}.
If an \piros{affine $F$-space} $V$ has a finite \piros{affine $F$-basis}, the\piros{n} all of its bases have the same number of element\piros{s}, the \piros{so-called (\emph{affine $F$-$)$ dimension}}  $\adim FV$  of the space. 
If $V$ is an \piros{affine $F$-space} with dimension $k$, then, for any $\set{b_0,\ldots,b_k}\subseteq V$,
\begin{equation}\label{spnKorbZs}
\set{b_0,\ldots,b_k}  \text{ spans }\aspace FV
\text{ if{f} } \set{b_0,\ldots,b_k}
\text{ is an \piros{affine $F$}-basis of } \aspace FV\text.
\end{equation}

\begin{lemma}\label{iHezTcXc}
 Let $L$ be a subfield of $\mathbb R$ such that $F\subseteq L$.
Assume that $X\subseteq F^n$. Then, for each 
$d\in F^n\cap \cnvclose LX$,  there are a $k\in\mathbb N_0$, an \piros{affine $L$-$($and therefore affine $F$-$)$} independent subset $\set{a_0,\ldots,a_k}$ of $X$, and  $\xi_0,\ldots,\xi_k\in\opset(F)$ such that $\xi_0+\cdots+\xi_k=1$ and $d=\xi_0a_0+\cdots+\xi_ka_k$. Consequently, $\cnvclose FX=F^n\cap\cnvclose LX$. 
\end{lemma}

\piros{
This lemma belongs to the folklore. For the reader's convenience (and having no reference at hand), we present a proof.}

\begin{proof}[Proof \piros{of Lemma~\ref{iHezTcXc}}]
Since $d\in \cnvclose L {X\cap L^n} \subseteq  \cnvclose{\mathbb R}{X\cap \mathbb R^n}$, we can choose an \piros{affine $\mathbb R$-subspace} $V\subseteq \mathbb R^n$ of minimal dimension such that $d\in \cnvclose {\mathbb R}{X\cap V}$. 
The \piros{affine $\mathbb R$-dimension} of $V$ will be denoted by $k$. 
By Carath\'eodory's Fundamental Theorem, there are $a_0,\ldots,a_k\in X\cap V$ such that $d\in\cnvclose {\mathbb R}{a_0,\ldots,a_k}$. The \piros{affine $\mathbb R$-subspace}  $\aspan{\mathbb R}{a_0,\ldots,a_k}$ is $V$ since otherwise a subspace with smaller dimension would do. 
Hence, using \eqref{spnKorbZs}, we conclude that  $\set{a_0,\ldots,a_k}$ is 
an \piros{affine $\mathbb R$-basis} of $V$.
Therefore, there is a unique $(\xi_0,\ldots,\xi_k)\in\mathbb R^{1+k}$ such that 
\begin{equation}\label{xirEeqs}
d=\xi_0a_0+\cdots+\xi_k a_k\text{ and }\xi_0+\cdots+\xi_k=1\text.
\end{equation}
These uniquely determined $\xi_i$ are non-negative since $d\in\cnvclose {\mathbb R}{a_0,\ldots,a_k}$. We can consider \eqref{xirEeqs} a system of linear equations for $(\xi_0,\ldots,\xi_k)$, and this system has a unique solution. Since $d,a_0,\ldots,a_k\in F^n$, 
the rudiments of  linear algebra imply that 
$(\xi_0,\ldots,\xi_k)\in F^{1+k}$.
This, together with the fact that the \piros{affine $\mathbb R$-independence} of the set $\set{a_0,\ldots,a_k}\subseteq F^n$ implies its \piros{affine $L$-independence}, proves the first part of the lemma. The second part is a trivial consequence of the first part.
\end{proof}

\begin{figure}
\centerline
{\includegraphics[scale=1.0]{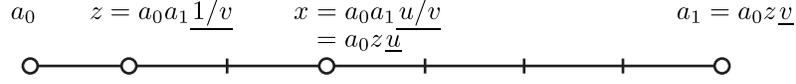}}
\caption{\piros{The} case $k=1$ and   $p=u/v=3/7$}\label{figtwo}
\end{figure}

\begin{proof}[Proof of Proposition~\ref{vknxLl}] 
\piros{Part \eqref{vknxLli} follows obviously from the fact that $a,b\in C$ with $a\neq b$ implies that $\lsegmnt Tabab \subseteq C$.} 

If $C$ is \piros{a} geometric \piros{convex subset of $F^n$}, then it is obviously $F$-convex. Conversely, if $C$ is $F$-convex,
then it is \piros{a} geometric \piros{convex subset of $F^n$} since  Lemma~\ref{iHezTcXc} yields that 
$C=\cnvclose FC=F^n\cap \cnvclose{\mathbb R}C$. This proves part \eqref{vknxLlii}.

Part \eqref{vknxLliii} is evident. 

In order to prove \eqref{vknxLliv}, assume that $C$ is $T$\piros{-segment convex}. Let $D:=\cnvclose{K} C$. Since $D$ is $K$-convex and $C\subseteq D$, it \piros{suffices} to show that $D\subseteq C$. Let $x$ be an arbitrary element of $D= \cnvclose{K} C$. 
We obtain from  Lemma~\ref{iHezTcXc}  that $D=K^n\cap\cnvclose{\mathbb R}C$. Hence, again by Lemma~\ref{iHezTcXc},  there are a \emph{minimal}  $k\in\mathbb N_0$, an \piros{affine $\mathbb R$-independent} subset $\set{a_0,\ldots,a_k}\subseteq C$, and a 
$(\xi_0,\ldots,\xi_k)\in \bigl(\opset(K)\bigr)^{1+k}$ such that 
\begin{equation*}
x=\xi_0a_0+\cdots+\xi_k a_k\text{ and }\xi_0+\cdots+\xi_k=1\text .
\end{equation*}
This allows us to prove the desired containment 
$x\in C$ by induction on $k$. If $k=0$, then $x=a_0\in C$ is evident. 

Next, assume that $k=1$. Then $x=a_0a_1\pop$ where $p=u/v\in \opint(K)$ and \piros{$u,v\in T$ with $0<u<v$}. 
Let $z:=a_0a_1\mop{1/v}$, see Figure~\ref{figtwo} for $u/v=3/7$, and we will rely on Lemma~\ref{barireveAl}. Then $\tline T{a_0}z$  contains  \piros{$a_0=a_0z\mop 0$ and} 
$a_1=a_0z\mop v$ since $\piros{0},v\in \piros{T}$. 
Hence $x=a_0z\mop u\in \lsegmnt T{a_0}z{a_0}{a_1}$ together with $T$-density implies that $x\in C$.

Finally, assume that $k>1$. Let $S$ denote the $k$-dimensional real simplex $\cnvclose{\mathbb R}{a_0,\ldots,a_k}$. 
The minimality of $k$ (or the positivity of all $\xi_i$) yields that $x$ is in the interior of  $S$. 
Hence the line $\tline{\mathbb R}{a_k}x$ has exactly two common points with the boundary of $S$; one of the\piros{m} is $a_k$, and  the other one will be denoted by $y$. 
Since $y$ can be obtained by solving a system of linear equations and all of $x,a_0,\ldots, a_k\in K^n$, it follows that $y\in K^n$.
Since $y$ belongs to a facet of $S$, the induction hypothesis yields that $y\in C$.
 Clearly, $x$ is between $a_k$ and $y$, whence $y=a_ky\pop$ for some $p\in\opint(K)$. Hence the already settled two-point case implies that $x\in C$.
\end{proof}
\sajat{REMARK: the proof worked even for $T=\mathbb Z$!}

The next lemma asserts that although $(C,\opset(T))$ cannot be generated by an independent set $G$  of points in general, $G$ satisfactorily describes $C$ by means of \emph{existential formulas}. This fact will enable us to use some ideas taken from \cite{MatRomSm:dyadpolyg}.

\begin{lemma}\label{fMlaLMam}
 Let $k\in  \mathbb N_0$ and $\xi_0,\ldots,\xi_k\in \mathbb Q$ such that $\xi_0+\cdots+\xi_k=1$. Then there exists \piros{an existential} formula $\keplet T{\xi_0,\ldots,\xi_k}{x_0,\ldots,x_k} y$ in the language of $(F^n,\opset(T))$ with the following property: whenever
\piros{$a_0,\ldots,a_k,b\in F^n$}, then 
\[b=\xi_0a_0+\cdots +\xi_ka_k\text{ if{f} } \keplet T{\xi_0,\ldots,\xi_k}{a_0,\ldots,a_k} b \text{ holds in } (F^n,\opset(T))  \text.
\]
If, in addition, $C$ is a $\mathbb Q$-convex subset of $F^n$ such that $\set{b,a_0,\ldots,a_k}\subseteq C$, then 
\[b=\xi_0a_0+\cdots +\xi_ka_k\text{ if{f} } \keplet T{\xi_0,\ldots,\xi_k}{a_0,\ldots,a_k} b \text{ holds in } (C,\opset(T))  \text.
\]
\end{lemma}

\begin{figure}
\centerline
{\includegraphics[scale=1.0]{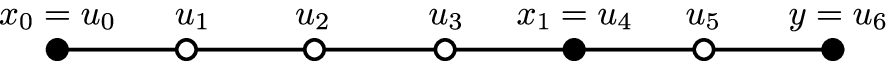}}
\caption{Illustrating $\keplet T{-2/4,\,6/4}{x_0,x_1}y$}\label{figthree}
\end{figure}

\begin{proof} Let $p$ be the smallest prime number such that $1/p\in T$; there is such a prime since $\mathbb Z\subset T\cap \mathbb Q$. 
We proceed by induction on $k$. If $k=0$, then $\xi_0=1$, so we let
$\keplet T{1}{x_0} y$ to be the formula $y=x_0$. 

Next, assume that $k=1$. To avoid a complicated formalization, we elaborate the details only when $(p,\xi_0,\xi_1)=(3,-1/2,3/2)$, see Figure~\ref{figthree}; the general case will clearly be analogous. The first step is to make all denominators equal and greater than $p$. Hence we write
$(3,-2/4,6/4)$ instead of $(3,-1/2,3/2)$. Then, as Figure~\ref{figthree} indicates, 
\begin{equation*}
\begin{aligned}
\keplet T{-2/4,\,6/4}{x_0,x_1}y:= (\exists u_0)\ldots(\exists u_6)\,\bigl(  x_0=u_0 \ees x_1=u_4 \ees u_0u_3\mop{1/3} =u_1\ees 
\cr
\kern 0cm  u_1u_4\mop{1/3} =u_2 \ees u_2u_5\mop{1/3} =u_3 \ees u_3u_6\mop{1/3} =u_4   \piros{{}\ees u_6u_3\mop{1/3} =u_5}     \ees y=u_6 \bigr)
\end{aligned}
\end{equation*}
does the job since $\set{b,a_0,a_1}\subseteq C$ together with the $\mathbb Q$-convexity of $C$ clearly implies that the $u_i$ belong to $C$.

Next, assume that $k\geq 2$ and the statement holds for smaller values. If one of the $\xi_0,\ldots,\xi_k$ is zero, say $x_i=0$, then we can obviously let
\[\keplet T{\xi_0,\ldots,\xi_k}{x_0,\ldots,x_k} y := \keplet T{\xi_0,\ldots,\xi_{i-1},\xi_{i+1} \ldots,\xi_k  }{x_0,\ldots,x_{i-1},x_{i+1},\ldots,x_k} y\text.
\]

So we can assume that none of the $\xi_i$ is zero. We have to partition $\set{0,1,\ldots,k}$ into the union of two  \piros{nonempty} disjoint subsets $I$ and $J$ such that  the $\xi_i$, $i\in I$, have the same sign, and the same holds for  the $\xi_j$, $j\in J$. If all the $\xi_0,\ldots,\xi_k$ are positive, then any partition will do. Otherwise we can let $\emptyset\neq I=\set{i: \xi_i<0}$; then $J=\set{0,\ldots,k}\setminus I$ is 
\piros{nonempty} since $\xi_0+\cdots+\xi_k=1>0$. To ease our notation, we can assume, without loss of generality, that $I=\set{0,\ldots,t}$ and $J=\set{t+1,\ldots,k}$. Let $\kappa_0=\xi_0+\cdots+\xi_t$ and $\kappa_1=\xi_{t+1}+\cdots+\xi_k$. Then $\kappa_0\neq 0\neq \kappa_1$ and $\kappa_0+\kappa_1=1$. Define $\eta_i:=\xi_i/\kappa_0$ for $i\leq t$ and  $\tau_j:=\xi_j/\kappa_1$ for $j > t$. Clearly, $\eta_0+\cdots+\eta_t=1$ and $\tau_{t+1}+\cdots+\tau_k=1$. Moreover, all the $\eta_i$ and the $\tau_j$ are positive, and the identity 
\[ \xi_0x_0+\cdots  + \xi_{k}x_{k} = \kappa_0( \eta_0x_0+\cdots +\eta_{t}x_{t})+ 
\kappa_1( \tau_{t+1}x_{t+1}+\cdots +\tau_{k}x_{k})
\]
clearly holds. Therefore we can let 
\begin{align*}
\keplet T{\xi_0,\ldots,\xi_k}{x_0,\ldots,x_k} y :&= \keplet T{\eta_0,\ldots,\eta_t}{x_0,\ldots,x_t}{z_0} \ees \keplet T{\tau_{t+1},\ldots,\tau_k}{x_{t+1},\ldots,x_k}{z_1}\cr
& \ees \keplet T{\kappa_0,\kappa_1}{z_0,z_1}y\text.
\end{align*}
This formula clearly does the job in $(F^n,\opset(T))$. It also works in \piros{$(C,\opset(T))$, provided that $C$ is $\mathbb Q$-convex,} since if $a_0,\ldots,a_k,b\in C$, then 
$\eta_0a_0+\cdots+\eta_ta_t\in C$ and 
$\tau_{t+1}a_{t+1}+\cdots+\piros{\tau}_ka_k\in C$, and the induction hypothesis (for $k-1$ and then for $k=1$) applies. 
\end{proof}

The following easy lemma is perhaps known (for arbitrary fields). Having no reference at hand, we will give \piros{an} easy proof.

\begin{lemma}\label{siGrTv} Let $C$ be an  $F$-convex subset of $F^n$. Assume that $\set{a_0,\ldots,a_k}$ is a maximal \piros{affine $F$-independent} subset of $C$, and let $V:=\aspan F{a_0,\ldots,a_k}$. Then 
\begin{enumeratei}  
\item\label{siGrTvi}  $C\subseteq V$ \piros{and $V=\aspan FC$}.
\item\label{siGrTvii} $V$ does not  depend on the choice of  $\set{a_0,\ldots,a_k}$.
\item\label{siGrTviii}  All maximal \piros{affine $F$-independent} subsets of $C$ consi\piros{st} of $1+k$ elements.
\end{enumeratei}
\end{lemma}

\begin{proof}
We know that $V=\{\xi_0a_0+\cdots+\xi_ka_k$ : $\xi_0+\cdots+\xi_k=1$, $(\xi_0,\ldots,\xi_k)\in F^{1+k}\}$. If we had $C\not\subseteq V$, then  $\set{a_0,\ldots, a_k,a_{k+1}}$ would be \piros{affine $F$-independent} for every $a_{k+1}\in C\setminus V$, which could contradict the maximality of $\set{a_0,\ldots,a_k}$. Hence $C\subseteq V$, \piros{which gives $\aspan FC\subseteq V$. Conversely,
$\set{a_0,\ldots,a_k}\subseteq C$ implies
that $V=\aspan F{a_0,\ldots,a_k}\subseteq \aspan FC$,}  proving part \eqref{siGrTvi}.  

Next, let  $\set{b_0,\ldots,b_t}$ be another maximal \piros{affine $F$-independent} subset of $C$, and let $W$ be the \piros{affine $F$-subspace} it spans. By part \eqref{siGrTvi}, $C\subseteq W$. Let $U:=V\cap W$. Since $C\subseteq U$,  $\set{a_0,\ldots,a_k}$ and $\set{b_0,\ldots,b_t}$ are \piros{affine $F$-independent} in $U$. This yields that $k\leq \adim FU$ and $t\leq \adim FU$. On the other hand, $U\subseteq V$ and $U\subseteq W$ give that $\adim FU\leq \adim FV=k$ and $\adim FU\leq t$. Hence $t=\adim FU =k$, proving part \piros{\eqref{siGrTviii}}. 

Using  $\adim FU=\adim FV$ and $U\subseteq V$ we conclude that  $U=V$. We obtain $U=W$ similarly, whence $W=V$ proves part \piros{\eqref{siGrTvii}}.
\end{proof}


\begin{proof}[Proof of Theorem~\ref{thmmain}]
Assume that \eqref{iimainitem} holds. Then $\psi$ is of the form $x\mapsto Ax+b$ where $b\in F^n$ \piros{is} a column vector and $A$ is an invertible $n$-by-$n$ matrix over $F$. Then $A$ is also an invertible real matrix and $b\in\mathbb R^n$, whence $\psi$ extends to an $\mathbb R^n\to\mathbb R^n$ automorphism. Thus, \eqref{iimainitem} implies \eqref{iiimainitem}.

Since $\opset(T)\subseteq \underline{\mathbb R}$, the automorphisms of the real affine space preserve the $\opset(T)$-structure. Hence  \eqref{iiimainitem} trivially implies 
 \eqref{imainitem}.

Next, assume that \eqref{imainitem} holds, and let $\phi\colon (C,\opset(T))\to(C',\opset(T))$ be an isomorphism. For $x\in C$, $\phi(x)$ will usually be denoted by $x'$. If an element of $C'$ is denoted by, say, $y'$, then $y$ will automatically stand for $\phi^{-1}(\piros{y'})$. We assume that $|C|>1$ since otherwise the \piros{statement} is trivial. 
Firstly, we show that 
\begin{equation}\label{sMedkjmon} \adim FC=\adim F{C'}\text.
\end{equation}
Since this is stipulated \piros{in the theorem} if $F\neq \mathbb Q$, 
let us assume that $F=\mathbb Q$ while proving \eqref{sMedkjmon}. Let, say $\adim{\mathbb Q}{C}\leq \adim{\mathbb Q}{C'}=:k$. By Lemma~\ref{siGrTv}, we can choose a (maximal) \piros{affine $F$-independent, that is $\mathbb Q$-independent,} subset $\set{a_0',\ldots,a'_k}$ in $C'$. It suffices to show that $\set{a_0,\ldots,a_k}\subseteq C$ is \piros{affine $F$-independent}. By way of contradiction, \piros{suppose} that this is not the case. Then, apart from indexing, there is a $t\in\set{1,\ldots,k}$ such that
$\set{a_1,\ldots,a_t}$ is \piros{affine $\mathbb Q$-independent} and  $a_0\in\aspan{\mathbb Q}{a_1,\ldots,a_t}$. Hence there are $\xi_1,\ldots,\xi_t\in\mathbb Q$ whose sum equals 1 such that $a_0=\xi_1a_1+\cdots+\xi_ta_t$. It follows from Lemma~\ref{fMlaLMam}
that $\keplet T{\xi_1,\ldots,\xi_t}{a_1,\ldots,a_t}{a_0}$ holds in $(C,\opset(T))$. Consequently, $\keplet T{\xi_1,\ldots,\xi_t}{a_1',\ldots,a_t'}{a'_0}$ holds in $(C',\opset(T))$.  Hence Lemma~\ref{fMlaLMam} implies that 
$a'_0=\xi_1a'_1+\cdots+\xi_ta'_t$, which contradicts the \piros{affine $F$-independence} of  $\set{a_0',\ldots,a'_k}$. This proves \eqref{sMedkjmon}.

\sajat{The paragraph starting here does NOT come from the paragraph above since $F$ can be different from $\mathbb Q$.}
\piros{Next, we} let $k=\adim FC=\adim F{C'}$. Clearly, $k\leq n$. Let $V:=\aspan FC$ and $V':=\aspan F{C'}$. We claim that for $t=0,1,\ldots, k$ and for an arbitrarily fixed $a_0\in C$,  
\begin{equation}\label{vanJoBzis}
\begin{aligned}
&\text{there are  }a_1,\ldots,a_t\in \piros C\text{ such that both } \set{a_0,\ldots,a_t}\subseteq C  \text{ and }\kern 0.1cm
\cr
&\set{a_0',\ldots, a_t'}=\phi\bigl( \set{a_0,\ldots, a_t} \bigr)
\subseteq C'\text{ are }\piros{\text{affine } F\text{-independent.}}
\end{aligned}
\end{equation}
\piros{(This assertion does not follow from the previous paragraph since here we do not assume that $F=\mathbb Q$.)}
Of course, we need 
\piros{\eqref{vanJoBzis}} only for $t=k$, but we prove it by induction on $t$. If $t\leq 1$, then \eqref{vanJoBzis} is trivial. Assume that $1<t\leq k$ and \eqref{vanJoBzis} holds for $t-1$. So we have an \piros{affine $F$-independent} subset $\set{a_0,\ldots,a_{t-1}}$ such that $\set{a_0',\ldots,a_{t-1}'}$  is also 
\piros{affine $F$-independent}. Let $\aspan F{a_0,\ldots,a_{t-1}}$ and $\aspan F{a'_0,\ldots,a'_{t-1}}$ be denoted by $V_{t-1}$ and $V'_{t-1}$, respectively. 
Since $t-1<k=\adim FC=\adim F{C'}$, there \piros{exist} element\piros{s} $x\in C\setminus V_{t-1}$ and  $y'\in C'\setminus V'_{t-1}$. 
Then $\set{a_0,\ldots,a_{t-1},x}$ 
and $\set{a'_0,\ldots,a'_{t-1},y'}$ are 
\piros{affine $F$-independent}. We can assume that $x'\in V'_{t-1}$ and $y\in V_{t-1}$   since otherwise $\set{a'_0,\ldots,a'_{t-1},x'}$ or 
$\set{a_0,\ldots,a_{t-1},y}$ would be \piros{affine $F$-independent}, and 
we could choose an appropriate $a_t$ from $\set{x,y}$.  
Take a $p\in\opint(T)$, and define $a_t:=yx\pop\in C$. Then $a'_t=y'x'\pop$. \piros{Suppose} for a contradiction that  $a_t\in V_{t-1}$.  Then, by Lemma~\ref{barireveAl},   $x=ya_t\mop{1/p}\in V_{t-1}$, a contradiction. Hence $a_t\notin V_{t-1}$ and $\set{a_0,\ldots, a_{t-1},a_t}$ is \piros{affine $F$-independent}. Similarly, 
\piros{suppose} for a contradiction that  $a'_t\in V'_{t-1}$.  Then, again by Lemma~\ref{barireveAl}, 
$y'=a'_t x' \mop{p/(p-1)}\in V'_{t-1}$ is a contradiction. Hence  $a'_t\notin V'_{t-1}$ and 
$\set{a_0,\ldots, a_{t-1},a'_t}$ is \piros{affine $F$-independent}.  This completes the proof of \eqref{vanJoBzis}.

\piros{From now on in the proof, \eqref{vanJoBzis} allows us to assume that  $\set{a_0,\ldots,a_k}\subseteq C$ and $\set{a'_0,\ldots,a'_k}\subseteq C'$ are affine $F$-independent subsets with $a_i'=\phi(a_i)$, for $i=0,\ldots,k$.}
For $\emptyset\neq  X\subseteq F^n$, we define
two ``relatively rational'' parts of $X$ as follows:  
\[
\rrps X :=X\cap\aspan{\mathbb Q}{a_0,\dots,a_k}\text{ and  }
\rrpv {X} :=X\cap\aspan{\mathbb Q}{a'_0,\dots,a'_k}\text{.}
\]
\piros{I}f $F=\mathbb Q$, then \piros{Lemma~\ref{siGrTv}\eqref{siGrTvi} yields that 
\[\rrps C=C\cap \aspan{\mathbb Q}{a_0,\ldots,a_k} = C\cap \aspan{\mathbb Q}C=C,
\]
and} $\rrpv{ C'}=C'$ \piros{follows similarly}.
Moreover, even if $F\neq \mathbb Q$, $\rrps C$ is dense in $C$, and $\rrpv {C'}$ is dense in $C'$ \piros{(in topological sense)}. 
The restriction of a map $\alpha$ to a subset $A$ of its domain will be denoted by $\restrict \alpha A$.
We claim that there is an automorphism $\psi$ of $\aspace F{F^n}$ such that 
\begin{equation}\label{psionDsEtoK}
\restrict{\psi}{\rrps{C}}= \restrict{\phi}{\rrps{C}}\quad\text{ and }\quad
\psi\bigl(\rrps C\bigr)=\rrpv {C'}\text.
\end{equation}
In order to prove this, 
extend $\set{a_0,\ldots,a_k}$ and $\set{a_0',\ldots,a_k'}$ to maximal \piros{affine $F$-independent} subsets $\set{a_0,\ldots,a_n}$
and $\set{a_0',\ldots,a_n'}$ of $\aspace F{F^n}$, respectively. Since $\set{a_0,\ldots,a_n}$
and $\set{a_0',\ldots,a_n'}$ are free generating sets of $\aspace F{F^n}$,  there is a (unique) automorphism $\psi$ of $\aspace F{F^n}$ such that $\psi(a_i)=a_i'$ for $i=0,\ldots, n$.

Let $x\in \rrps C$ be arbitrary. Then there are 
$\xi_0,\ldots,\xi_k\in \mathbb Q$ such that their sum equals 1 and  
\begin{equation}\label{cxaibLcB}
x=\xi_0a_0+\ldots+\xi_ka_k\text.
\end{equation}
\piros{Observe that $C$ and $C'$ are $\mathbb Q$-convex since they are $F$-convex. Hence}
we obtain from Lemma~\ref{fMlaLMam} and \piros{\eqref{cxaibLcB}} that 
$\keplet T{\xi_0,\ldots,\xi_k}{a_0,\ldots,a_k}x$ holds in $(C,\opset(T))$. Since $\phi$ is an isomorphism,  $\keplet T{\xi_0,\ldots,\xi_k}{a_0',\ldots,a_k'}{\phi(x)}$ holds in $(C',\opset(T))$. Using Lemma~\ref{fMlaLMam} again, we conclude that $\phi(x)=\xi_0a_0'+\ldots+\xi_ka_k'$. Therefore, \eqref{cxaibLcB} yields that 
$\psi(x)= \xi_0\psi(a_0)+\ldots+\xi_k\psi(a_k)=  \xi_0a_0'+\ldots+\xi_ka_k'=\phi(x)\in C'$. 
This gives that $\restrict{\psi}{\rrps{C}}= \restrict{\phi}{\rrps{C}}$ and  $\psi(x)\in\rrpv {C'}$. Therefore, $\psi\bigl(\rrps C\bigr)\subseteq \rrpv {C'}$. Working with $(\psi^{-1},\phi^{-1})$ instead of $(\psi,\phi)$, we obtain $\psi^{-1}\bigl(\rrpv {C'}\bigr)\subseteq \rrps {C}$ similarly. Thus, \eqref{psionDsEtoK} holds.

If $F=\mathbb Q$, then  \eqref{psionDsEtoK} together with $C=\rrps C$ and  $C'=\rrpv {C'}$ implies the validity of the theorem. 
\piros{Thus} we assume that \piros{at least one of} $C$ and $C'$ \piros{is} bounded. 
\piros{If, say, $C$ is bounded, then so is $\rrps C$. The automorphisms of $\aspace F{F^n}$ preserve this property, whence \eqref{psionDsEtoK} implies that $\rrpv {C'}$ is bounded. Since $\rrpv {C'}$ is dense in $C'$, we conclude that $C'$ is bounded. Therefore, in the rest of the proof, we assume that both $C$ and $C'$ are bounded.}

For $X\subseteq \mathbb R^n$, the topological closure of $X$, that is the set of cluster points of $X$, will be denoted by $\topclose{\mathbb R}X$. 
Let $\pontf C=\psi^{-1}(C')$. It is an $F$-convex subset of $F^n$ since the automorphisms  of $\aspace F{F^n}$ are also automorphisms of $(F^n,\opset(F))$.
By the same reason,  the restriction $\restrict{\psi^{-1}}{C'}$
is an isomorphism $(C',\opset(T))\to(\pontf C,\opset(T))$. Let $\gamma:=\restrict{\psi^{-1}}{C'}\circ \phi$ (we compose maps from right to left). Then, by \eqref{psionDsEtoK}, by $\gamma(a_i)=a_i$ for $0\leq i\leq n$, and by Lemma~\ref{siGrTv}, we know that 
\begin{equation}\label{soTgxML}
\begin{aligned}
&\gamma\colon (C,\opset(T))\to (\pontf C,\opset(T))\quad \text{is an isomorphism,}
\cr
&\rrps C=\rrps{\pontf C},\quad  \text{and}\quad 
\restrict \gamma{\rrps C}\text{ is the identical map,}
\cr
&C \subseteq V:=\aspan F{a_0,\ldots,a_k}\quad \text{and}\quad \pontf C\subseteq V\text.
\end{aligned}
\end{equation}
It suffices to show that $\gamma$ is the identical map; \piros{really}, then the desired $\phi=\restrict \psi C $ would follow by the definition of $\gamma$. 
For $y\in C$, \piros{the element} $\gamma(y)$ will often \piros{be} denoted by $\pontf y$. 
We have to show that $\pontf y=y$ for all $y\in C$. Since this is clear by \eqref{soTgxML}
if $y\in\rrps C$, we assume that 
\[\text{$y\in C\setminus\rrps C$.}\] 
Next, we deal with $C$ and $\pontf C$ simultaneously. Since they play a symmetric role,  we usually give the details only for $C$.

If $\vec b=(b_1,b_2,b_3, \ldots)\in \rrps C^\omega= \rrps {\pontf C}^\omega$, then $\vec b$ is called an $\rrps C$-sequence. Convergence (without adjective) is understood in the usual sense in $\mathbb R^n$. We use the notation $\lim_{j\to\infty}b_j=y$ to denote that $\vec b$ converges to $y$.  We say that $\vec b\,\,$ \emph{$(C,\opset(T))$-converges} to $y$, 
in notation $\vec b\toT C y$, if for each $j\in\mathbb N$,
\begin{equation}\label{convDef}
\text{there exist an } x_j\in C \text{ and a } q_j\in\opint(T) \text{ such that }q_j\leq 1/j
\text{ and } b_j= yx_j\mop{q_j}\text.
\end{equation}
In virtue of Lemma~\ref{barireveAl}, $\vec b\toT C y$ if{f}  
\begin{equation}\label{convvarDef}
\text{for each }j\in\mathbb N,
\text{ there is a } q_j\in\opint(T) \text{ such that }q_j\leq 1/j\text{ and } yb_j\mop{1/q_j}\in C\  \text.
\end{equation}
It \piros{follows} from \eqref{soTgxML} and  \eqref{convDef} that  for all $\vec b\in {\rrps C}^\omega$, 
\begin{equation}\label{trNsFercoNv}
\vec b\toT C y\quad\text{if{f}}\quad \piros{\vec b}    \toT {\pontf C} \pontf y\text.  
\end{equation}
For $X\subseteq \mathbb R^n$, let  $\diam X$ denote the \emph{diameter} $\sup\set{\distance uv: u,v\in X}$ of $X$. We know that
$\diam C<\infty$ and $\diam{\pontf C}<\infty$. 
Hence if $q_j\leq 1/j$, then Lemma~\ref{barireveAl} yields  that 
$ \distance y{b_j}= q_j\cdot \distance y{yb_j\mop{1/q_j}} \leq \diam C/j$.  Hence \eqref{convvarDef} gives that  for any $\rrps C$-sequence $\vec b$, 
\begin{equation}\label{CconvThNnconV}
\begin{aligned}
&\text{if}\quad  \vec b\toT C y,\quad\text{then}\quad \lim_{j\to \infty} b_j= y\text.\quad\text{Similarly,} 
\cr
&\text{if}\quad  \vec b\toT {\pontf C} \pontf y,\quad\text{then}\quad \lim_{j\to \infty} b_j=\pontf y\text.  
\end{aligned}
\end{equation}

Next, we intend to show that 
\begin{equation}\label{vanbSqv}
\text{there exists a }\rrps C\text{-sequence }\vec b\text{ such that }
\vec b\toT C y\text.
\end{equation}
Extend $\set y$ to a maximal \piros{affine $F$-independent} subset $\set{y,z_1,\ldots,z_k}$ of $C$. It follows from Lemma~\ref{siGrTv} that this set consists of $1+k$ elements, and $V$ equals $\aspan F {y,z_1,\ldots,z_k}$\piros{.}
For a \piros{given} $j\in\mathbb N$, choose a $q_j\in\opint(T)$ such that $q_j\leq 1/j$. For $i=1,\ldots,k$, let $u_i:=yz_i\mop{q_j}$. By the $F$-convexity of $C$, $u_i\in C$. 
Since $z_i=yu_i\mop{1/q_j}$ by Lemma~\ref{barireveAl},  $\set{y,u_1,\ldots,u_k}$ also $F$-spans $V$, whence it is \piros{affine $F$-independent by Lemma~\ref{siGrTv}\eqref{siGrTviii}}. \piros{Hence} $\cnvclose F {y,u_1,\ldots,u_k}\subseteq C$ is a (non-degenerate) $k$-dimensional simplex of $V$, so its interior (understood in $V$) is nonempty. Since $\rrps C$ is dense in $C$ and $\rrps C\subseteq C\subseteq V$, we can choose a point $b_j\in \cnvclose F {y,u_1,\ldots,u_k}$. By \eqref{HoWgEnre}, $b_j$ is of the form $yu_1\ldots u_k\,\ttau$. Let $x_{\piros j}:=yz_1\ldots z_k\,\ttau\in C$. 
Using that $\moop{q_j}$ commutes with $\ttau$ and the \piros{terms} are idempotent, we have that 
\begin{align*}
yx_{\piros j} \mop{q_j}&=y( yz_1\ldots z_k\,\ttau) \mop{q_j} = (yy\ldots y\,\ttau) ( yz_1\ldots z_k\,\ttau)\mop{q_j} 
\cr
&=(yy\mop{q_j})(yz_1\mop{q_j})\ldots(yz_k\mop{q_j})\ttau = yu_1\ldots u_k\ttau=b_j\text.
\end{align*}
(Notice that the parentheses above can be omitted.) Therefore, the sequence $\vec b=(b_1,b_2,\ldots)$ proves \eqref{vanbSqv}.

Finally, it follows from \eqref{vanbSqv}, \eqref{trNsFercoNv} and \eqref{CconvThNnconV} that $\pontf y=y$. \piros{Therefore, $\gamma$ is the identical map.}
\end{proof}

\begin{proof}[Proof of Corollary~\ref{coroLegy}] As we have already mentioned, with reference to \cite{RS02}, $(F^n,\hoop)$ is term equivalent to $(F^n,\opset(\mathbb D))$. Hence the first part of the statement is clear.

To prove the second part, assume that $D,D'$ are \piros{isomorphic subalgebras} of \piros{$(\mathbb Q^n,\hoop)$ such that}  $D'$ is \piros{a} geometric \piros{subset of $\mathbb Q^n$}. Then  there is an isomorphism $\phi\colon (D,\opset(\mathbb D))\to  (D',\piros{\opset}(\mathbb D))$ by the already mentioned term equivalence, and $D'$ is $\mathbb Q$-convex by Proposition~\ref{vknxLl}\piros{\eqref{vknxLlii}}. We have to show that $D$ is $\mathbb Q$-convex. Let $a,b\in D$ and $q\in \opint(\mathbb Q)$. It is clear from  
\piros{Lemma~\ref{fMlaLMam}} that 
\begin{align*}
ab\qop\in D &\iff
(\exists y)\, \keplet {\mathbb D}{1-q,\,q}{a,b}y \text{ holds in }(D,\opset(\mathbb D))\cr
&\iff 
(\exists y')\, \keplet {\mathbb D}{1-q,\,q}{\phi(a),\phi(b)}{y'} \text{ holds in }(D',\opset(\mathbb D)),
\end{align*}
and this last condition holds since $D'$ is $\mathbb Q$-convex.
\end{proof}

%
%
%
%
%
%
%
%
%

\end{document}
\bye